\documentclass[12pt,a4paper]{amsart}

\RequirePackage[ansinew]{inputenc}
\usepackage[english]{babel}
\usepackage{amsmath,amsfonts,amsthm,amssymb}
\usepackage[top=3cm,left=2.5cm,right=3cm,bottom=3cm]{geometry}
\usepackage[usenames,dvipsnames]{color}
\usepackage[colorlinks=true,linkcolor=blue,citecolor=blue]{hyperref}

\newcommand{\C}{\mathbb{C}}
\newcommand{\R}{\mathbb{R}}

\newcommand{\del}{\partial}
\newcommand{\delbar}{\bar{\partial}}
\newcommand{\ddc}{\text{dd}^c}

\newtheorem{theorem}{Theorem}[section]
\newtheorem{proposition}[theorem]{Proposition}
\newtheorem{corollary}[theorem]{Corollary}
\newtheorem{lemma}[theorem]{Lemma}

\theoremstyle{definition}
\newtheorem{definition}[theorem]{Definition}

\newtheorem{remark}[theorem]{Remark}

\numberwithin{equation}{section}

\title{Some properties of plurisubharmonic functions}
\author{Alano Ancona}
\address{Département de Mathématiques (UMR 8628), Bat. 425, Université Paris-Sud, Faculté des Sciences d'Orsay, F-91405 Orsay Cedex, France}
\email{alano.ancona@math.u-psud.fr}
\author{Lucas Kaufmann}
\address{Sorbonne Universités, UPMC Univ Paris 06, IMJ-PRG, UMR 7586 CNRS, Univ Paris Diderot, Sorbonne Paris Cité, F-75005, Paris, France}
\email{lucas.kaufmann@imj-prg.fr}
\date{}

\begin{document}
\maketitle

\begin{abstract}
Two properties of  plurisubharmonic functions are proven. The first result is a Skoda type integrability theorem with respect to a Monge-Ampère mass with Hölder continuous potential. The second one says that  locally, a  p.s.h.\ function is   $k$-Lipschitz outside a set of Lebesgue measure smaller that $c/k^2$.
\end{abstract}

\section{Introduction and main results}

Let $\Omega$ be an open subset of $\C^n$. Recall that a  function $\varphi: \Omega \to  [- \infty , \infty )$ is  plurisubharmonic\footnote{We adopt here the standard definition given in \cite{demailly:agbook} and \cite{hormander:convexity} and don't  exclude functions that are identically $-\infty$ in some connected component of $\Omega$. }  (p.s.h.\ for short) if $\varphi$ is upper semicontinuous and if for  every complex line $L \subset \C^n$ the function $\varphi|_{\Omega \cap L}$ is subharmonic in $\Omega \cap L$.

Basic examples  are given by $\varphi := \log |h|$ with $h:\Omega  \to  { \mathbb C} $  holomorphic in $\Omega$, in particular $ \varphi  (z)=\log \vert  z \vert  $ with $\Omega ={ \mathbb C} ^n$. Plurisubharmonicity is preserved on taking the maxima of a  finite number of p.s.h.\ functions and on  taking  the pointwise limit of a decreasing sequence of p.s.h.\ functions.

The \textit{Lelong number} of a p.s.h.\ function $\varphi:\Omega     \to  [- \infty , \infty )$, $\Omega  \subset   { \mathbb C} ^n$,  at  $a \in \Omega$ is defined as $$\nu(\varphi;a) := \liminf_{z \to a,\,z\ne a}\frac{\varphi(z)}{\log|z-a|}$$
and it somewhat measures the singularity of $\varphi$ at $a$. This number can be characterized as $\nu(\varphi;a) = \sup \{\gamma: \varphi(z) \leq \gamma \log |z-a| + O(1) $ as $z \to  a\,\}$ and  is one of the most basic quantities associated to the singularity of $\varphi$ at $a$. The function $z \mapsto \nu(\varphi;z)$ is upper semicontinuous with respect to the usual topology. It is a deep theorem of Y-T. Siu that this function is also upper semicontinuous with respect to the Zariski topology, i.e. for evey $c > 0$ the set  $ \{ a \in \Omega \,;\, \nu ( \varphi  ,a) \geq c\, \}$ is a closed analytic subvariety of $\Omega $. For further properties and equivalent definitions of the Lelong number the reader may consult \cite{demailly:agbook} and \cite{hormander:convexity}.

A classical theorem of Skoda \cite{skoda:sous-ens-an} states that if $\nu(\varphi;a)<2$ then $e^{-\varphi}$ is integrable in a neighborhood of $a$ with respect to the Lebesgue measure. This result is basic, for instance, in the study of multiplier ideal sheaves associated to a p.s.h.\ function (see \cite{lazarsfeld-positivity2}).

Our first result is a generalization of Skoda's theorem where the Lebesgue measure is replaced by a general Monge-Ampère mass with a Hölder continuous local potential (see section \ref{sec:positive-currents} for the definitions). This class of measures appears naturally in the study of holomorphic dynamical systems (see for instance \cite{sibony:dynamique-Pk}).

\begin{theorem} \label{thm:integrability}
Let $0<\alpha \leq 1$, let $u:\Omega  \to  (- \infty , \infty )$ be an $\alpha$-Hölder continuous p.s.h.\ function in the domain $\Omega  \subset    \C^n$ and let $z \in   \Omega $. If $\varphi$ is a p.s.h. function in $\Omega $ and if  $\nu(\varphi;z) < \frac{2 \alpha}{\alpha + n(2-\alpha)}$, then there is a neighborhood $K \subset   \Omega $ of $z$ such that the integral
\begin{equation*}
\int_K e^{-\varphi} (\ddc u)^n
\end{equation*}
is finite. In other words, $e^{- \varphi  }$ is locally integrable in $U:= \{  \xi   \in   \Omega \, ;\, \nu(\varphi; \xi  ) < \frac{2 \alpha}{\alpha + n(2-\alpha)} \,\} $ with respect to the positive measure $(\ddc u)^n$.
\end{theorem}

Our second result is independent of the first and it applies to a wider class of functions  than the class of p.s.h.\ functions. In what follows, a function $F:\Omega \to  [- \infty , \infty )$ defined on an open subset $\Omega $ of $ \C^n$ is said to be  \emph{separately subharmonic} if  for every $\xi=( \xi  _1,\dots,  \xi _n) \in { \mathbb C} ^n$ and every $j=1,\ldots,N$ the partial function $z  \to  F(\xi_1,\dots,\xi_{j-1},z,$ $\xi_{j+1}\dots,\xi_n)$ is subharmonic in its domain of definition (i.e.\, $U_ \xi := \{ z \in   { \mathbb C} \,;(\xi_1,\dots,\xi_{j-1},$ $z,\xi_{j+1}\dots,\xi_n) \in   \Omega  \, \}$).

Roughly speaking, the result says  that given ${ \mathbb K}  \subset    \subset   \Omega $ and $ \varepsilon >0$, there is  a set $L \subset  { \mathbb K} $ whose Lebesgue measure is $ \leq  \varepsilon $ such that  $F$ is $k$-Lipschitz in $\mathbb K \setminus L$ with $k\sim  \varepsilon ^{-2}$. A precise statement is as follows.

\begin{theorem} \label{thm:lipschitz-dimN} Let $F$ be a negative separately subharmonic function on a connected open subset $\Omega $ of $ \C^n$. Let $\omega $ and $\omega '$ be two non-empty open and relatively compact subsets of $\Omega $.
Then for every real number $k >0$,  there is a compact set $L \subset   \omega  $ such that  

(i) $F_{ | L}$ is finite and $k$-Lipschitz

(ii) $|\omega  \setminus L | \leq { \frac {  C} {k^2 }} \,  | F( \xi  ) | ^2$ for every $ \xi   \in   \omega ' $,

where  $C$ is a positive constant depending only on $\Omega $, $\omega $ et $\omega '$.
\end{theorem}

{\bf Acknowledgements.} Both authors would like to warmly thank T-C. Dinh  whose questions are at the origin of the main results or this paper. The second author was supported by a grant from Région Île-de-France.

\section{Preliminary material} \label{sec:introduction}

\subsection{Closed positive currents and Monge-Ampère measures} \label{sec:positive-currents}

A $k$-current on a complex manifold $X$ of dimension $n$ is a continuous linear form on the space of compactly supported differential forms of degree $(2n-k)$. Such objects generalize $k$-forms with coefficients in $L^1_{loc}$ and submanifolds of real codimension $k$. For the basic theory of currents in complex manifolds see \cite{demailly:agbook}.

The existence of a complex structure implies, by duality, that every $k$-current decomposes as a sum of $(p,q)$-currents with $p+q = k$. Real currents of type $(p,p)$ carry a notion of positivity (see \cite{demailly:agbook}, \cite{lelong:positivity}). Examples of positive currents include Kähler forms and currents of integration along complex submanifolds of $X$.

The operators $\del$, $\delbar$, $\text d$ and $\text d ^c = \frac{i}{2\pi}(\delbar - \del)$ extend to currents by duality and the $\ddc$-Poincaré Lemma states that every positive closed $(1,1)$-current $T$ can be written locally as $T = \ddc u$ where $u$ is a p.s.h.\ function, called the local potential of $T$.

If $T$ is a closed positive current and $u$ is a locally bounded p.s.h.\ function the current $uT$ is well-defined and  the product (or intersection) current $\ddc u \wedge T$ can be defined  by the formula $$\ddc u \wedge T \stackrel{\text{def}}{=} \ddc(uT).$$

One may then define the product $S \wedge T$ when $S$ is a closed positive $(1,1)$-current with bounded local potential: just write  $S = \ddc u$ locally and use the above formula. However,  the wedge product of general currents may not be well defined.

By induction, the product $\ddc u_1 \wedge \ldots \wedge \ddc u_p$ is well-defined when  $u_1,\ldots,u_p$ are locally bounded p.s.h.\ functions. In particular, if $u$ is a locally bounded p.s.h.\ function then the current $(\ddc u)^n$ is well-defined. It is a positive measure called \emph{Monge-Ampère measure} associated to $u$. See \cite{klimek} or the original paper \cite{bedford-taylor} for some basic properties of the Monge-Amp\`ere operator.

\subsection*{Locally moderate currents}
Recall that the set of all p.s.h.\ functions in $X$ is closed in the space $L^1_{loc}(\Omega )$  and that every family of p.s.h.\ functions that is bounded in $L^1_{loc}(\Omega )$ is relatively compact in $L^1_{loc}(\Omega )$ (see Theorem 3.2.12 in \cite{hormander:convexity}). For the sake of simplicity, such a family is called a compact family.

The notion of locally moderate currents and measures was introduced by Dinh-Sibony (see \cite{dns:estimates} for more details). 

\begin{definition}
A measure $\mu$ on a complex manifold $X$ is called \textit{locally moderate} if for any open set $U \subset X$, any compact set $K \subset U$ and any compact family $\mathcal F$ of p.s.h. functions on $U$  there are constants $\beta> 0$ and $C> 0$ such that
\begin{equation*}
\int_K e^{-\beta \psi} d\mu \leq C,\;\; \text{for every }\; \psi \in \mathcal F.
\end{equation*}
\end{definition}

It  follows immediately from the definition that for any $\mathcal F$ and $\mu$ as above, $\mathcal F$ is bounded in $L^p_{loc}(\mu)$ for $1\leq p < \infty$ and that $\mu$ does not charge pluripolar sets.

A positive closed current $S$ of type $(p,p)$ on $X$ is said to be \textit{locally moderate} if the trace measure $\sigma_S = S \wedge \omega^{n-p}$ is locally moderate. Here $n = \dim X$ and $\omega$ is the fundamental form of a fixed Hermitian metric on $X$. 

\begin{theorem}(Dinh-Nguyen-Sibony \cite{dns:estimates}) \label{thm:holder-moderate}
If $1\leq p \leq n$ and $u_1,\ldots,u_p$ are Hölder continuous p.s.h. functions on $X$ then the Monge-Ampère current $\ddc u_1 \wedge \ldots \wedge \ddc u_p$ is locally moderate.
\end{theorem}

The proof in  \cite{dns:estimates} uses the following two lemmas. They will also be used here to prove Theorem \ref{thm:integrability}. We denote $B_r$ the ball of radius $r$ centered at the origin of $\C^n$ and fix a fundamental form $\omega $ as above. 

\begin{lemma} (\cite{dns:estimates}) \label{lemma:loc-bdd}
Let $S$ be a locally moderate closed positive current of type $(n-1,n-1)$ on $B_r$. If $\mathcal G$ is a compact family of p.s.h. functions on $B_r$ then $\mathcal G$ is bounded in $L^1_{loc}(\sigma_S)$. Moreover, the mass of the measures $\ddc \varphi \wedge S$, $\varphi \in \mathcal G$ are locally bounded in $B_r$ uniformly on $\varphi$.
\end{lemma}

\begin{proof}(as in \cite{dns:estimates}) Let $K$ be a compact subset of $B_r$. After subtracting a fixed constant we may assume that every element of $\mathcal G$ is negative on $K$. Since $\sigma_S$ is locally moderate we can choose $\beta,C >0 $ such that  $\int_K e^{-\beta \varphi} d\sigma_S \leq C$ for every $\varphi \in \mathcal G$. We thus have $\int_K \beta |\varphi| \; d\sigma_S \leq  \int_K e^{-\beta \varphi} d\sigma_S \leq C$ for every $\varphi \in \mathcal G$ which proves the first assertion.

For the second assertion let $K$ be a compact subset of $B_r$ and consider a cut-off function $\chi$ which is equal to  $1$ in a neighborhood of $K$ and which is supported on a larger compact $L \subset B_r$. We have, for $\varphi \in \mathcal G$  $$ \int_K \ddc \varphi \wedge S \leq \int_L \chi  \ddc \varphi \wedge S = \int_L \ddc \chi \wedge \varphi  S \leq \|\chi\|_{\mathcal C ^2} \int_L |\varphi| \; d\sigma_S,$$
which is uniformly bounded by the first part of the lemma.
\end{proof}

\begin{lemma} (\cite{dns:estimates}) \label{lemma:int-parts}
Let $r > 0$, $S$ be a locally moderate closed positive current of type $(n-1,n-1)$ on $B_{2r}$ and $u$ be an $\alpha$-Hölder continuous p.s.h. function on $B_r$ which is smooth on $B_r \setminus B_{r-4\rho}$ for some $0<\rho<r / 4$. Fix a smooth cut-off function $\chi$ with compact support in $B_{r-\rho}$, $0 \leq \chi \leq 1$ and $\chi \equiv 1$ on $B_{r-2\rho}$.

If $\varphi$ is a p.s.h. function on $B_{2r}$, then
\begin{equation*}
\begin{split}
\int_{B_r} \chi \varphi \ddc(uS) = -\int_{B_r \setminus B_{r-3\rho}} \ddc \chi \wedge \varphi u S - \int_{B_r \setminus B_{r-3\rho}} \text{d} \chi \wedge \varphi \text{d}^c u \wedge S \\
+ \int_{B_r \setminus B_{r-3\rho}} \text{d}^c \chi \wedge \varphi \text{d}u \wedge S + \int_{B_{r-\rho}} \chi u \; \ddc \varphi \wedge S.
\end{split}
\end{equation*}
\end{lemma}

Notice that  the smoothness of $u$ in $B_r \setminus B_{r-4\rho}$ makes the second and third integrals meaningful.
\begin{proof} (ref \cite{dns:estimates})
The case when $\varphi$ is smooth follows from a direct computation using integration by parts. The general case follows by approximating $\varphi$ by a decreasing sequence of smooth p.s.h.\ functions. See \cite{dns:estimates} for the complete proof.
\end{proof}

 We will also need a volume estimate of the sublevel sets of p.s.h.\ functions due to M. Kiselman.  We include Kiselman's argument here for the the reader's convenience.

\begin{lemma} (\cite{kiselman:sous-niveau}) \label{lemma:sub-level-volume}
Let $\varphi$ be a p.s.h.\ function on an open set $\Omega \subset \C^n$ and $K \subset \Omega$ be a compact subset. Then, for every $\gamma < 2 / \sup_{z \in K} \nu(\varphi;z)$ there is a constant $C_\gamma = C_\gamma(\varphi,\Omega,K)$ such that
\begin{equation*}
\lambda(K \cap \{\varphi \leq -M \}) \leq C_\gamma \,e^{-\gamma M}, \;\; M \in \R,
\end{equation*}
where $\lambda$ denotes the Lebesgue measure in $\C^n$.
\end{lemma}

\begin{proof}
Since $e^{\gamma (-M - \varphi)} \geq 1$ on $K \cap \{\varphi \leq -M \}$ we have
\begin{equation*}
\mu(K \cap \{\varphi \leq -M \}) \leq \int_K e^{\gamma (-M - \varphi(z))} d \lambda (z) = e^{-\gamma M} \int_K e^{-\gamma \varphi} d \lambda.
\end{equation*}

It suffices then to take $C_\gamma = \int_K e^{-\gamma \varphi} d \lambda$, which is finite by Skoda's Theorem since $\nu(\gamma \varphi; z) < 2$ for every $z \in K$.
\end{proof}

\subsection{Maximal functions and regularity in $W^{1,1}$}

If $U$ is open in $\R^N$ and if $f \in  L^1_{loc} (U)$, the Lebesgue set of $f$ is $${ \mathcal L}_f= \left \{ x \in   U\,:\,  \exists  a _0 \in   \R \text{ such that }\; \lim_{r \to 0} \oint _{B(x,r)} | f(t)-a_0 | \; dt=0\; \right\},$$ where the sign $\oint _A$ denotes the average over the set $A$. When  $a_0=a_0(x)$ exists it is equal to  $ \tilde f(x):=\lim_{ r \to 0} \oint_{ B(x,r) } f(x)\, dx $ and it is well know that (i) ${ \mathcal L}_f$ is a Borel set, (ii)\, $ \lambda _N(U\setminus { \mathcal L}_f)=0$ and (iii) $ \tilde f(x)=f(x)$ a.e.\ in $U$.  

For  a function $f \in L^1_{loc}(\R^N)$,  the Hardy-Littlewood  maximal function of $f$  is denoted
$${ \mathcal M}_f(a):=\sup_{r>0} \oint_{B(a,r)}\,  | f (x) | \, dx,\ \ \ a \in   \R^N,$$

More generally, for an open set $U \subset \R^N$ and $f \in L^1_{loc}(\R^N)$ we set $${ \mathcal M}^{\, \rho} _ f(a):=\sup_{0<r< \rho } \oint_{B(a,r)}\,  |f (x) | \, dx,\ \ \  \rho >0,\; a \in   U_ \rho, $$ 
where $U_ \rho = \{ z \in   U\,;\, d(z, U^c) > \rho \; \})$. It is easily checked that ${ \mathcal M}^{\, \rho} _ f$ is Borel measurable in $U_ \rho $.

We may now recall three classical results that will be basic for us in the next section. 

\begin{theorem} \label{thm:bojarski} (Bojarski \cite{bojarski}, Bojarski-Haj\l asz \cite{bojarski-hajlasz:sobolev})
Let $f \in   W^{1,1}_{\rm loc} (U)$. If $x \in   U_ \rho $  is such that ${ \mathcal M}^{ \rho }_{ | \nabla f | } (x)<+ \infty $ then $x$ is a Lebesgue point for $f$. Furthermore, for every $x$, $y \in   U_ \rho $ such that $ | x-y |  \leq {\frac  {  \rho } {3}}  $ and ${ \mathcal M}^{ \rho }_{ | \nabla f | }(x)< \infty $, ${ \mathcal M}^{ \rho }_{ | \nabla f | }(y)< \infty $ we have   $$|  \tilde f(x)- \tilde f(y) |  \leq C_N\;  | x-y | \; ({ \mathcal M}^{ \rho }_{ | \nabla f | }(x)+{ \mathcal M}^{ \rho }_{ | \nabla f | }(y)),$$
where $C_N$ is a constant depending only on the dimension $N$.
\end{theorem}

\begin{theorem} (Hardy-Littlewood, Wiener. Ref.\ \cite{adams-hedberg}) Let  $\mu$ be a finite positive measure on $\R^N$ and let ${ \mathcal M}(\mu  )(x)=\sup_{B(a,r)\ni x} { \frac {  \mu  (B(a,r))} { | B(a,r) | } }$, $x \in   \R ^N$. Then for every $t>0$$$ |  \{  \mathcal M (\mu  )>t\,  \} |   \leq C\,  { \frac {  \;  \|\mu \|_1} {t }},$$
where $C$ depends only on the dimension $N$.
\end{theorem}


Consider now for $0< \alpha <N$,  the Riesz kernel $I_\alpha (x) = |x|^{\alpha - N}$  of order $\alpha$ in $\R^N$. For a finite Radon measure on $\R^N$ the Riesz potential $I_\alpha(\mu) $ is defined by $I_ \alpha (\mu  )(x):=I_ \alpha \ast  \mu  (x)=\int  | x-y | ^{ \alpha -N}\, d\mu  (y)$ for $x \in   { \mathbb R} ^N$.

\begin{theorem} \label{thm:riesz-estimate} (Zygmund, cf. \cite{adams-hedberg} p.\ 56) Let $\mu  $ be a finite positive Radon measure on $\R^N$ and assume that $ 0<\alpha  < N$. Then there is a constant $A$ depending only on $\alpha$ and $N$ such that, for every $t>0$,
$$ |  \{I_ \alpha (\mu) \geq t \} |  \leq  \frac { A} {t ^{\frac{N}{N- \alpha} }} \|\mu \|_1^{\frac{N}{N- \alpha}}.$$
\end{theorem}

In fact we  need a slightly improved version of this estimate with $I_ \alpha (\mu  )$  replaced by its maximal function. 

\begin{theorem} \label{thm:maximal-riesz-estimate}
Let $\mu  $ be a finite positive measure on $\R^N$,  $ 0<\alpha <N$,  and  let ${ \mathcal I}_ \alpha \mu  (x)=\mathcal M_{I_ \alpha (\mu)  }(x)$ be the maximal function of $I_ \alpha (\mu  )$. Then there is a constant $A$ depending only on $\alpha$ and $N$ such that, for every $t>0$,
\begin{equation}
|  \{  \mathcal I_ \alpha (\mu  ) \geq t \} |  \leq \,{ \frac {  A} {t ^{\frac{N}{N- \alpha} }}}\; \|\mu \|_1^{\frac{N}{N- \alpha}}.
\end{equation}
\end{theorem}

\begin{proof} {\bf A.} We first note an  elementary fact. Denote
$$\chi_r= | B(0,r) | ^{-1} \,\mathbf 1_{B(0,r)}.$$

For  $0<s \leq r$  we have  $\chi_s\ast \chi_r \leq \,2^N\,\chi _{2r}.$ Indeed, letting $V_N$ to denote the volume of the unit ball in $\R^N$, we have  $\chi _r \leq V_N^{-1}r^{-N}$ and then $\chi_s\ast \chi_r \leq V_N^{-1}r^{-N}$ because $\chi_s$ is of integral $1$.
 On the other hand, $\chi_s\ast \chi _r$ vanishes outside  $ B(0,2r)$ and $\chi_{2r}=V_N^{-1} (2r)^{-N}$ in $B(0,2r)$, from where  the stated inequality follows.

In particular we have $\mu  \ast \chi_r\ast \chi _s \leq 2^ N \, \mu  *\chi_{2r} \leq  2^ N\, { \mathcal M}(\mu  )$ for $0<s \leq r$.  Thus using the commutativity of the convolution,  $\mu  \ast \chi_r\ast \chi _s (x)\, \leq 2^ N \, {\mathcal M}(\mu  )(x)$ for every $r,\, s>0$ and on taking the supremum over $s$ we obtain $${ \mathcal M}(\mu \ast \chi_r  )  \leq 2^ N\, { \mathcal M}(\mu  ).$$

Observe that  $I_ \alpha (\mu  )\ast \chi_r=I_ \alpha *(\mu  *\chi_r)$.
\medskip

{\bf B.} Next we adapt the argument of Hedberg's proof of Theorem \ref{thm:riesz-estimate} (see \cite{adams-hedberg} p. 56). Dividing the integral defining $I_ \alpha (\mu  )(x)$ in two parts  we get (exactly as in Hedberg's proof) 
\begin{align}
 \nonumber     I_ \alpha (\mu) (x)&=  \int _{y\notin B(x, \delta )} {\frac  { d\mu  (y) }{ | x-y | ^{N-\alpha }}} +\int _{y \in   B(x, \delta )}  {\frac  { d\mu  (y) }{ | x-y | ^{N-\alpha }}} \\
   \nonumber & \leq  { \frac {  \mu  (B(x, \delta ))} {  \delta ^{N- \alpha }  }} \,+\, (N- \alpha ) \int _0^ \delta \, { \frac {  \mu  (B(x,t))} {t^{N+1- \alpha } }}\, dt \;+ \;  A\, \delta ^ {\alpha -N}\, \|\mu   \|_1 \\
 \nonumber & \leq  \; \delta ^ \alpha \, { \mathcal M}(\mu  )(x)+ \;  A' \delta ^ \alpha \, { \mathcal M}(\mu  )(x)+ \;  A\, \delta ^ {\alpha -N} \, \|\mu  \|_1 \\
 \nonumber & \leq  C\,  ( \delta ^ \alpha \, { \mathcal M}(\mu  )(x)+ \;   \delta ^{\alpha -N} \, \|\mu  \|_1).
 \end{align}

So we have
\begin{align}
   \nonumber    I_ \alpha (\mu)\ast \chi _r (x)& 
\leq  C\,  ( \delta ^ \alpha \, { \mathcal M}(\mu\ast \chi_r  )(x)+ \;   \delta ^{\alpha -N} \, \|\mu \ast \chi_r  \|_1)
\\ \nonumber & \leq C'\,  ( \delta ^ \alpha \, { \mathcal M}(\mu  )(x)+ \;   \delta ^{\alpha -N} \, \|\mu \|_1),\end{align}
where we used the inequality  from  part A.

Taking the supremum over $r$ we obtain $$\mathcal I_ \alpha (x) = { \mathcal M}(I_ \alpha (\mu  ))(x) \leq C'\,  ( \delta ^ \alpha \, { \mathcal M}(\mu  )(x)+ \;   \delta ^{\alpha -N} \, \|\mu  \|_1),$$  
and setting the constant to be $ \delta := (\|\mu \|_1 \,/ \,{ \mathcal M}(\mu  )(x))^{ \frac {  1} {N }}$ we get
$$ \mathcal I_ \alpha (x) \leq  C\,  \|\mu \|_1^{ \alpha /N} \; ({ \mathcal M}(\mu  )(x))^{1-( \alpha /N)}.$$

Finally,  by the Hardy-Littlewood-Wiener Theorem,
\begin{align}
   \nonumber    |  \{ \mathcal I_\alpha  >t \} |  & \leq  |  \{ C\,  \|\mu \|_1^{ \alpha /N} \; ({ \mathcal M}(\mu))^{1-( \alpha /N)} >t\,  \} |
\\ \nonumber & =  |  \{ ({ \mathcal M}(\mu))^{1-( \alpha /N)} >C^{-1}\,t\, \|\mu \|_1^{ -\alpha /N}  \} |
\\ \nonumber &  =    |  \{{ \mathcal M}(\mu) >C'\, t^{ \frac {  N} {N- \alpha  }}\, \|\mu \|_1^{ -\alpha /N- \alpha }  \} |
\\ \nonumber & \leq  \frac { A} {t ^{\frac{N}{N- \alpha} }} \|\mu \|_1^{\frac{N}{N- \alpha}}.
 \end{align} 
\end{proof}

\section{Integration of p.s.h.\ functions}

This section is devoted to the proof of Theorem \ref{thm:integrability} and some related results.

We will need the following simple extension  of the second part of Lemma \ref{lemma:loc-bdd}.

\begin{lemma} \label{lemma:loc-bdd2}
Let $1\leq p \leq n$ and let  $S$ be a locally moderate closed positive current of type $(n-p-1,n-p-1)$ on $B_r$. If $\mathcal G$ is a compact family of p.s.h. functions on $B_r$ and $\mathcal H$ is locally uniformly bounded  family of  p.s.h.\ functions on $B_r$ then the mass of the measures $\ddc \varphi \wedge (\ddc u)^p \wedge S$, $\varphi \in \mathcal G, u \in \mathcal H$ are locally bounded in $B_r$ uniformly on $\varphi$ and $u$.
\end{lemma}

\begin{proof}
Fix  a compact subset $K$  of $B_r$ and let $L_0 = K, L_1, \ldots,  L_{p}$ be compact subsets of $B_r$ such that $L_i$ is contained in the interior of $L_{i+1}$. Let $\chi_i$, $i=1,\ldots, p$ be smooth cut-off  functions such that $0\leq \chi_i \leq 1$, $\chi_i \equiv 1$ in $L_{i-1}$ and $\chi_i$ is supported in $L_{i}$. Then, for $\varphi \in \mathcal G$ and $u \in \mathcal H$, the mass of $\ddc \varphi \wedge (\ddc u)^p \wedge S$ over $K$ is bounded by
\begin{equation*}
\begin{split}
 \int_{L_1} \chi_1 \ddc \varphi \wedge (\ddc u)^p \wedge S &= \int_{L_1} u (\ddc \chi_1) \wedge \ddc \varphi \wedge (\ddc u)^{p-1} \wedge S \\
&\leq \|\chi_1\|_{\mathcal C^2} \|u\|_{L^\infty(L_1)} \int_{L_1} \ddc \varphi \wedge (\ddc u)^{p-1} \wedge S \wedge \omega \\
&\leq \|\chi_1\|_{\mathcal C^2} \|u\|_{L^\infty(L_1)} \int_{L_2} \chi_2 \ddc \varphi \wedge (\ddc u)^{p-1} \wedge S \wedge \omega \\
&=  \|\chi_1\|_{\mathcal C^2} \|u\|_{L^\infty(L_1)} \int_{L_2} u\, \ddc \chi_2 \wedge \ddc \varphi \wedge (\ddc u)^{p-2} \wedge S \wedge \omega\\
&\leq \|\chi_1\|_{\mathcal C^2} \|\chi_2\|_{\mathcal C^2} \|u\|_{L^\infty(L_1)} \|u\|_{L^\infty(L_2)} \int_{L_2} \ddc \varphi \wedge (\ddc u)^{p-2} \wedge S \wedge \omega^2 \\
& \leq \ldots \\
&\leq \|\chi_1\|_{\mathcal C^2} \cdots \|\chi_p\|_{\mathcal C^2} \|u\|_{L^\infty(L_1)} \cdots \|u\|_{L^\infty(L_p)} \int_{L_p} \ddc \varphi  \wedge S \wedge \omega^p,
\end{split}
\end{equation*}
where $\omega= \ddc \|z\|^2 $ is the standard fundamental form on $ {\mathbb C }^n$. The result now follows from Lemma \ref{lemma:loc-bdd} and the fact that $\|u\|_{L^\infty(L_i)}$ is bounded independently of $u$.
\end{proof}

\begin{proof}[Proof of Theorem \ref{thm:integrability}]
There is no loss of generality in assuming that $z = 0$ and since $\varphi$ is locally bounded from above we may also assume that $\varphi$ is negative. As before $\omega = \ddc \|z\|^2 $.

The proof is inspired by the methods in  \cite{dns:estimates}, Theorem 1.1. It will consist of  successive applications of integration by parts formulas (Lemma \ref{lemma:int-parts}) together with a regularization procedure.

For $N> 0$ define $\varphi_N = \max \{\varphi, - N\}$ and $\psi_N = \varphi_{N-1} - \varphi_N$. Notice that $0 \leq \psi_N \leq 1$, $\psi_N$ is supported in $\{\varphi < -N + 1\}$ and $\psi_N \equiv 1$ in $\{\varphi < -N \}$.

Observe that
\begin{equation} \label{eq:first-estimate}
\begin{split}
\int e^{-\varphi} (\ddc u)^n &= \sum_{N=0}^\infty \int_{\{-N \leq \varphi < -N +1\}} e^{-\varphi} (\ddc u)^n \leq \sum_{N=0}^\infty e^N \int_{\{-N \leq \varphi < -N +1\}} (\ddc u)^n \\
&\leq \sum_{N=0}^\infty e^N \int \psi_{N-1} (\ddc u)^n.
\end{split}
\end{equation}

From the hypothesis that $\nu(\varphi;0) < \frac{2 \alpha}{\alpha + n(2-\alpha)}$ and from the upper semicontinuity of the function $z \mapsto \nu(\varphi;z)$ there is an $r>0$ such that $\sup_{z \in B_{2r}} \nu ( \varphi  ,z)\leq \frac{2 \alpha}{\alpha + n(2-\alpha)} - \sigma$ for a small constant $\sigma > 0$. From Lemma \ref{lemma:sub-level-volume} we get that
\begin{equation} \label{eq:volume-estimate}
\lambda(B_{2r} \cap \{\varphi \leq -N + 1\}) \lesssim e^{-(\frac{\alpha + n(2-\alpha)}{\alpha} + \delta) N} = e^{-(1+\delta)N} e^{-\frac{n(2-\alpha)}{\alpha} N},
\end{equation}
where $\delta > 0$ is a small constant (depending on $\varphi $). Here and in what follows the sign $\lesssim$ means that the left-hand sign is smaller or equal than a constant times the right-hand side, the constant being independent from $N$. 

Taking a smaller $r$ if necessary we may assume that $u$ is defined on $B_{2r}$. Subtracting a constant we may assume that $u\leq -1$. Consider the function $v(z) = \max(u(z),A \log \|z\|)$. If we choose $A > 0$ sufficiently small, we see that $v$ coincides with $u$ near the origin and that $v(z) = A \log \|z\|$ near the boundary of $B_r$. This allows us to assume that $u(z) = A \log \|z\|$ on $B_r \setminus B_{r-4\rho}$ for some fixed $\rho < r/4$. Notice that, in particular, $u$ is smooth on $B_r \setminus B_{r-4\rho}$.

Fix a smooth cut-off function $\chi$ with compact support in $B_{r-\rho}$, $0 \leq \chi \leq 1$ and $\chi \equiv 1$ on $B_{r-2\rho}$. Applying Lemma \ref{lemma:int-parts} to $\psi_{N-1}$ and $(\ddc u)^{n-1}$ and noticing that $(\ddc u)^n = \ddc(u (\ddc u)^{n-1})$ we get
\begin{equation} \label{eq:step1}
\begin{split}
&\int_{B_r} \chi \psi_{N-1} (\ddc u)^n = \\ - &\int_{B_r \setminus B_{r-3\rho}} \ddc \chi \wedge \psi_{N-1} u (\ddc u)^{n-1} - \int_{B_r \setminus B_{1-3\rho}} \text{d} \chi \wedge \psi_{N-1} \text{d}^c u \wedge (\ddc u)^{n-1} \\
+& \int_{B_r \setminus B_{r-3\rho}} \text{d}^c \chi \wedge \psi_{N-1} \text{d}u \wedge (\ddc u)^{n-1} + \int_{B_{r-\rho}} \chi u \; \ddc \psi_{N-1} \wedge (\ddc u)^{n-1}.
\end{split}
\end{equation}

Observing that $u$ is smooth in $B_r \setminus B_{r-3\rho}$, that the support of $\psi_{N-1}$ is contained in $\{\varphi \leq -N + 1\}$ and using the volume estimate (\ref{eq:volume-estimate}) we get that the absolute values of the first three integrals on the right-hand side are $\leq c_1 e^{-(1+\delta)N} e^{-\frac{n(2-\alpha)}{\alpha} N}$, where $c_1>0$ does not depend on $N$.

For $N\geq 1$ set $\varepsilon = \varepsilon(N) =  e^{-(\frac{1}{\alpha} + c)N}$, where $0<c<\frac{\delta}{n(2- \alpha)}$. Using a convolution with a smooth $U(n)$-invariant approximation of identity one can choose for $N$ large a regularization $u_\varepsilon$ of $u$ defined on $B_{r-\rho}$ in such a way that $\|u - u_\varepsilon\|_\infty \lesssim \varepsilon^\alpha = e^{-(1+c \alpha)N}$ and $\|u_\varepsilon\|_{\mathcal C^2}:= \|u_\varepsilon\|_{\mathcal C^2(B_{r-\rho })}\lesssim \varepsilon^{\alpha-2}$.

Writing $u = u_\varepsilon + (u - u_\varepsilon)$ the last integral in (\ref{eq:step1}) is equal to
\begin{equation*}
\int_{B_{r-\rho}} \chi u_\varepsilon \; \ddc \psi_{N-1} \wedge (\ddc u)^{n-1} + \int_{B_{r-\rho}} \chi (u - u_\varepsilon) \; \ddc \psi_{N-1} \wedge (\ddc u)^{n-1}.
\end{equation*}

Since $\{\varphi_N\}_{N\geq 0}$ is a compact family of p.s.h.\ functions and since the current $(\ddc u)^{n-1}$  is locally moderate (Theorem \ref{thm:holder-moderate}), we see from Lemma \ref{lemma:loc-bdd}  that the modulus ot the second integral above is less than $c_2 \|u - u_\varepsilon\|_\infty \leq c'_2 \, e^{-(1+c \alpha)N}$ where $c'_2>0$ does not depend on $N$.

To deal with the remaining integral we apply Lemma \ref{lemma:int-parts} for $u_\varepsilon$ instead of $u$. Noticing that $\ddc (u_\varepsilon \wedge (\ddc u)^{n-1}) = \ddc u_\varepsilon \wedge (\ddc u)^{n-1}$ we get
\begin{equation*}
\begin{split}
&\int_{B_{r-\rho}} \chi u_\varepsilon \; \ddc \psi_{N-1} \wedge (\ddc u)^{n-1} = \\ &\int_{B_r \setminus B_{r-3\rho}} \ddc \chi \wedge \psi_{N-1} u_\varepsilon (\ddc u)^{n-1} + \int_{B_r \setminus B_{r-3\rho}} \text{d} \chi \wedge \psi_{N-1} \text{d}^c u_\varepsilon \wedge (\ddc u)^{n-1} \\
- &\int_{B_r \setminus B_{r-3\rho}} \text{d}^c \chi \wedge \psi_{N-1} \text{d}u_\varepsilon \wedge (\ddc u)^{n-1} + \int_{B_r} \chi \psi_{N-1} \; \ddc u_\varepsilon \wedge (\ddc u)^{n-1}.
\end{split}
\end{equation*}

Since $u(z) = A \log \|z\|$ on $B_r \setminus B_{r-4\rho}$ the $\mathcal C^2$ norm of $u_\varepsilon$ on $B_r \setminus B_{r-3\rho}$ does not depend on $\varepsilon = \varepsilon(N)$. Together with the volume estimate (\ref{eq:volume-estimate}) this implies that the first three integrals in the right-hand side have absolute values less than  $ c_3 \, e^{-(1+\delta)N} e^{-\frac{n(2-\alpha)}{\alpha} N}$ where $c_3>0$ does not depend on $N$.

For the last integral we write $\ddc u_\varepsilon \wedge (\ddc u)^{n-1} = \ddc (u (\ddc u_\varepsilon \wedge (\ddc u)^{n-2})$ and apply Lemma \ref{lemma:int-parts} for $S=\ddc u_\varepsilon \wedge (\ddc u)^{n-2}$. This gives us four integrals. Three of them are integrals over $B_r \setminus B_{r-3\rho}$ involving $u$, $u_\varepsilon$, $\psi_{N-1}$ and its derivatives. As above, the absolute value of each one of  them is $\lesssim e^{-(1+\delta)N} e^{-\frac{n(2-\alpha)}{\alpha} N}$. The remaining integral is
\begin{equation*}
\int_{B_{r-\rho}} \chi u \ddc \psi_{N-1} \wedge \ddc u_\varepsilon \wedge (\ddc u)^{n-2},
\end{equation*}
which we write again as
\begin{equation} \label{eq:step3}
\int_{B_{r-\rho}} \chi u_\varepsilon \ddc \psi_{N-1} \wedge \ddc u_\varepsilon \wedge (\ddc u)^{n-2} + \int_{B_{r-\rho}} \chi (u - u_\varepsilon) \ddc \psi_{N-1} \wedge \ddc u_\varepsilon \wedge (\ddc u)^{n-2}.
\end{equation}

Since $u_\varepsilon$ converges to $u$ in $L^\infty$, Lemma \ref{lemma:loc-bdd2} implies that the mass of $\ddc \psi_{N-1} \wedge \ddc u_\varepsilon \wedge (\ddc u)^{n-2}$  is bounded independently of $N$ and $\varepsilon$. Therefore, the modulus of the second integral above is less than $c_4 \|u - u_\varepsilon\|_\infty \leq c'_4 \, e^{-(1+c \alpha)N}$ where $c'_4>0$ does not depend on $N$.

To deal with the the first integral in (\ref{eq:step3}) we apply Lemma \ref{lemma:int-parts}, obtaining three integrals over $B_r \setminus B_{r-3\rho}$ whose absolute values are $\lesssim e^{-(1+\delta)N} e^{-\frac{n(2-\alpha)}{\alpha} N}$ and the integral
\begin{equation*}
\int_{B_{r-\rho}} \chi \psi_{N-1} (\ddc u_\varepsilon)^2 \wedge (\ddc u)^{n-2}.
\end{equation*}

We can repeat the above procedure in order ``move'' the $\ddc$'s from $u$ to $u_\varepsilon$. We  get at each step integrals with absolute values $\lesssim e^{-(1+\delta)N} e^{-\frac{n(2-\alpha)}{\alpha} N}$ or $\lesssim e^{-(1+c\alpha)N}$ (where the constants involved don't depend on $N$) and at the final step we get the integral
\begin{equation*}
\int_{B_{r-\rho}} \chi \psi_{N-1} (\ddc u_\varepsilon)^n,
\end{equation*}
whose absolute value is less than $ c_5 \|u_\varepsilon\|_{\mathcal C^2}^n \cdot \lambda\{\varphi \leq -N + 1\} \leq c'_5 \,\varepsilon^{n(\alpha - 2)} e^{-(1+\delta)N} e^{-\frac{n(2-\alpha)}{\alpha} N} = c'_5 \, e^{(cn(2-\alpha) - (1+ \delta))N}$, with $c'_5$ independent from $N$.

Altogether the above estimates yield
\begin{equation*}
\int_{B_{r-\rho}} \chi \psi_{N-1} (\ddc u)^n \lesssim e^{-(1+\delta)N} e^{-\frac{n(2-\alpha)}{\alpha} N} +  e^{-(1+c\alpha)N} + e^{(cn(2-\alpha) - (1+ \delta))N} .
\end{equation*}

Inserting  these estimates  in (\ref{eq:first-estimate}) we finally get
\begin{equation*}
\begin{split}
\int_{B_r} e^{-\varphi} (\ddc u)^n \lesssim \sum_{N=0}^\infty e^N\left [e^{-(1+\delta)N} e^{-\frac{n(2-\alpha)}{\alpha} N} +  e^{-(1+c\alpha)N} + e^{(cn(2-\alpha) - (1+ \delta))N} \right ] \\
= \sum_{N=0}^\infty \left [e^{-\delta N-\frac{n(2-\alpha)}{\alpha} N} + e^{-c \alpha N} + e^{(cn(2-\alpha) - \delta)N} \right ].
\end{split}
\end{equation*}

By the choice of $c$ all the factors of $N$ in the exponentials above are negative, so the series converges and hence the integral $\int e^{-\varphi} (\ddc u)^n$ is finite.
\end{proof}

From Theorem \ref{thm:integrability} and a computation analogous to the one made in the proof of Lemma \ref{lemma:sub-level-volume} follows an estimate of the measure of the sub-level sets of p.s.h.\ functions with respect to Monge-Ampère masses with Hölder continuous potential.

\begin{corollary}
Let $\varphi$ be a p.s.h. function on an open set $\Omega \subset \C^n$ and $\mu = (dd^c u)^n$ a Monge-Ampère mass on $\Omega$ with $u$ an $\alpha$-Hölder continuous p.s.h.\ function. If $K \subset \Omega$ is a compact subset then for every $\gamma < \frac{2 \alpha}{\alpha + n(2-\alpha)} \frac{1}{ \sup_{z \in K} \nu(\varphi;z)}$ there is a constant $C_\gamma = C_\gamma(\varphi,\Omega,K)$ such that
\begin{equation*}
\mu(K \cap \{\varphi \leq -M \}) \leq C_\gamma e^{-\gamma M}, \;\; M \in \R.
\end{equation*}
\end{corollary}

Another theorem of Skoda concerns the non-integrability of a p.s.h.\ function with large Lelong number: if $\nu(\varphi;0) > 2n $ then $e^{-\varphi}$ is not integrable in any neighborhood of the origin with respect to Lebesgue measure (see \cite{hormander:convexity} Lemma 4.3.1). One cannot hope for a similar result with respect to every Monge-Ampère measure with Hölder continuous potential, because the measure $\mu = (\ddc u)^n$ can be arbitrarily small near $0$ (and even zero), making the integral $\int e^{-\varphi} d\mu$ finite. We may note however  the following fact.

\begin{proposition} \label{prop:non-integrabilty}
Fix $0<\alpha \leq1$.  There exists a Monge-Ampère mass $\mu = (\ddc u)^n$ where $u$ is an $\alpha$-Hölderian p.s.h.\ function such that for every p.s.h.\ function $\varphi$ defined near  $0$ with $\nu(\varphi;0) > n \alpha $ we have $\int_K e^{-\varphi} d\mu = + \infty$ for every neighborhood $K$ of the origin. 
\end{proposition}

\begin{proof}
Let $\gamma = \nu(\varphi;0) > n \alpha$. Since $\varphi(z) \leq \gamma \log \|z\| + O(1)$ near $0$  (see Introduction) we have that $e^{-\varphi(z)} \geq C \frac{1}{\|z\|^{\gamma}} \geq C \frac{1}{\|z\|^{n \alpha}}$. If we take $u(z) = \|z\|^\alpha$ then a direct computation shows that $(\ddc u)^n = C^{st} \|z\|^{n(\alpha - 2)} \cdot \lambda$ in the sense of currents. We thus have
\begin{equation*}
\int e^{-\varphi} (\ddc u)^n \geq C^{st} \int \frac{1}{\|z\|^{n\alpha}} \frac{1}{\|z\|^{n(2 -\alpha)}}d\lambda = C^{st} \int \frac{1}{\|z\|^{2n}}d\lambda,
\end{equation*}
and the last integral diverges in any neighborhood of the origin.

\end{proof}

\begin{remark} \label{rmk:n=1-sharp}
For $n=1$ the condition on the Lelong number of $\varphi$ on the hypothesis of Theorem \ref{thm:integrability} is $\nu(\varphi;z) < \alpha$. This bound is sharp as Proposition \ref{prop:non-integrabilty} shows.



\end{remark}

\begin{remark} \label{rmk:logz1}
For $n \geq 2$ the condition $\nu(\varphi;z) < \frac{2\alpha}{\alpha + n (2-\alpha)}$ in Theorem \ref{thm:integrability} is probably no longer optimal, as the example below suggest.

Let $0<\alpha \leq 1$ and $c>0$. Consider the potential $u(z) = |z_1|^\alpha + \cdots + |z_n|^\alpha$ and the p.s.h.\ function $\varphi(z) = c \log |z_1|$ defined on $\C^n$. We have then that $(\ddc u)^n = n! \left(\frac{\alpha}{2}\right)^{2n} |z_1|^{\alpha-2}  \cdots |z_n|^{\alpha-2} \cdot \lambda$ as measures on $\C^n$. Notice that this expression makes sense, since $\alpha > 0$ implies $|z_1|^{\alpha-2}  \cdots |z_n|^{\alpha-2} \in L^1_{loc}(\C^n)$.

We thus have
\begin{equation*}
\begin{split}
\int_{B_1} e^{-\varphi} (\ddc u)^n &= n!  \left(\frac{\alpha}{2}\right)^{2n} \int_{B_1} \frac{1}{|z_1|^c}  |z_1|^{\alpha-2}  \cdots |z_n|^{\alpha-2} \; d\lambda \\
&= n! \left(\frac{\alpha}{2}\right)^{2n} \int_{B_1} \frac{1}{|z_1|^{c-\alpha +2}}  |z_2|^{\alpha-2}  \cdots |z_n|^{\alpha-2} \; d\lambda,
\end{split}
\end{equation*}
which is finite if and only if $\nu(\varphi;0) = c < \alpha$.

\end{remark}

\section{Regularity of p.s.h.\ functions}

This section is devoted to the proof Theorem \ref{thm:lipschitz-dimN}. Since our approach is rather in the spirit of classical potential theory it will be convenient to deal with \emph{(positive) superharmonic functions}, instead of (negative) subharmonic functions. We refer the reader to \cite{book-armitage-gardiner} \cite{livre-brelot},  \cite{book-helms} for the basic potential theoretic notions  used in what follows.

Let us start with the one complex dimensional case, that is, that of superharmonic functions on an open subset of $\C$. 

\begin{theorem} \label{thm:lipschitz-dim1} Let $F$ be a positive superharmonic function on a connected open subset $\Omega $ of $ \C $ and let ${\mathbb K} $ be a relatively compact subset of $\Omega $. Let $\omega $ be a fixed open neighborhood of ${\mathbb K} $  which is relatively compact in $\Omega $ and set $I_{F,\omega }=\min  \{ F(x)\,:\, x \in   \omega  \}$.

Then for every real number $k >0$ there is a compact set $L \subset   {\mathbb K} $ such that  

(i) $F_{ | L}$ is finite and $k$-Lipschitz

(ii)  $| {\mathbb K} \setminus L | \leq { \frac {  C} {k^2 }}\,(I_{F,\omega })^2$,

where $C$ is a positive constant depending only on $\Omega $, $\omega $ and ${\mathbb K} $.
\end{theorem}

\begin{remark} \label{rmk:lebesgue-set}
Since $F$ is superharmonic,  $F(a)=\lim_{r \to  0} \oint _{B(a,r)}\, F(x)\, dx$ for  $a \in   \Omega $ and so $ \tilde F=F$ on the Lebesgue set ${ \mathcal L}_F$. Since $F$ is l.s.c.\  it follows that ${ \mathcal L}_F=\Omega \cap \{ F<+ \infty \}$.
\end{remark}

\begin{proof} We can assume that  $\Omega $ is bounded and replacing $F$ by its reduction (or r\'{e}duite)  over  $\omega$  -with respect to $\Omega $- (see \cite{livre-brelot}, \cite{book-helms} , \cite{book-armitage-gardiner} or the proof of Lemma \ref{lemma:integral-function-comparison}) we can assume that $F$ is a potential in $\Omega$, i.e., the greatest harmonic minorant of $F$ in $\Omega$ is zero, and that  $\mu = - \Delta F$ is a positive measure  supported in  $\overline\omega $. Thus, $F$ is positive and harmonic in $\Omega \setminus \overline \omega $ and  vanishes on  $\partial \Omega $.

Since $F=G(\mu  )$ is the Green potential of $\mu  $ in $\Omega $ and $G(z,w) \geq c$ in $\overline \omega \times \overline  \omega$ we have that $\|\mu \|_1 \leq C^{ste}\,I_{F,\omega }$. Therefore, in order to prove the theorem it suffices to find $L$ such that $F$ is $k$-Lipschitz on $L$ and $ | \mathbb K \setminus L | \leq { \frac {  C} {k^2 }} \, \|\mu   \|_1^2$.

Write  $F=N*\mu  +H$ where $N(z)= \frac {1} {2 \pi} \log \frac {1} {\|z\|}$ and $H$ is a harmonic function on  $\Omega $. Setting  $R:=\sup  \{   | z-z' | \,;\, z \in   \overline \omega,\; z' \in   \partial \Omega \, \}$ and  $r:=\inf  | z-z' | \,;\, z \in   \overline \omega,\; z' \in   \partial \Omega \, \}$ we have over $\partial\Omega $ the inequalities
$$\frac {1} {2 \pi} \log \left( \frac { 1} {R} \right) \,  \| \mu \|_1 \leq N*\mu   \leq \frac {1} {2 \pi} \log \left( \frac { 1} {r} \right) \,  \| \mu \|_1,$$
which implies that $H$ is bounded by two fixed multiples of $\|\mu\|_1$ in $\Omega $. By the Harnack property (or the Poisson formula), we conclude that for $\omega '$ a relatively compact open subset $\Omega $ we have $\|\nabla H \|_{L^\infty (\omega') } \leq C \|\mu  \| _1$  for a constant $C>0$ depending only on $N$, $\Omega $ et $\omega '$. Taking $\omega '$ to be a connected neighborhood of $\overline \omega $ we see that $H$ is $c \|\mu \| _1$-Lipschitz over $\omega $ where $c=c(N,\omega ',\omega )$. It suffices then to prove (i) and (ii) for $s:=N*\mu  $ instead of $F$.

Notice that $s \in   W^{1,p} _{loc}(\Omega )$ for $1\leq p <  2$ and $ \frac {\partial s} {\partial x_j }= \frac {  \partial N} {\partial x_j } * \mu  =-{ \frac {  1} {2 \pi   }}{ \frac {  x_j} { | x | ^2 }}\, *\, \mu  $. From the fact that $ | { \frac {  x_j} { | x | ^2 }}  | \leq { \frac {  1} { | x |  }}$ and Theorem \ref{thm:maximal-riesz-estimate} it follows that the maximal function of $ | \nabla s | $ satisfies the weak  type $L^2$ inequalities,  
\begin{equation} \label{eq:weakL2}
|\{\mathcal M_{ | \nabla s | }\geq t\,  \} | \leq \frac {C} {t^2 }\,  \|\mu   \|_1 ^{\, 2},\;\; t>0. 
\end{equation}

Given $k>0$,  let $A = \{\mathcal M_{|\nabla s |} \leq \frac{k}{2C_2} \}$, where $C_2$ is the constant appearing in Theorem \ref{thm:bojarski}. From this theorem and Remark \ref{rmk:lebesgue-set} we get $$ | s(z)-s(w) |  \leq k | z-w | \text{ for every  } z,\, w \in   A,$$ and by  (\ref{eq:weakL2}) $ | \C \setminus A | \leq { \frac {  C} {k^2 }} \, \|\mu   \|_1^2$.  Since $ \lambda _2$ is inner regular, the proof is complete.

\end{proof}

We now proceed to the proof of Theorem \ref{thm:lipschitz-dimN}, stated in terms of separately superharmonic functions.

\begin{theorem} Let $F$ be a positive separately superharmonic function on a connected open subset $\Omega $ of $ \C^n$. Let $\omega $ and $\omega '$ be two non-empty open and relatively compact subsets of $\Omega $.
Then for every real number $k >0$,  there is a compact set $L \subset   \omega  $ such that  

(i) $F_{ | L}$ is finite and $k$-Lipschitz

(ii) $|\omega  \setminus L | \leq { \frac {  C} {k^2 }} \,  | F( \xi  ) | ^2$ for every $ \xi   \in   \omega ' $,

where  $C$ is a positive constant depending only on $\Omega $, $\omega $ et $\omega '$.
\end{theorem}

\begin{remark} \label{rmk:avanissian}
By a theorem of Avanissian \cite{avanissian}, we know that $F$ is lower semicontinuous and superharmonic in $\Omega $.
\end{remark}   

\begin{proof}  We may assume $F \not \equiv + \infty $. The result being local  (cf.  Lemmas \ref{lemma:integral-function-comparison} and \ref{lemma:lipschitz-gluing}  below) we can assume $\omega =\omega ' = D(0,1)^n$ and $\Omega =D(0,4)^n$.

{\bf A.} For a fixed $  \xi  :=( \xi  _1, \xi  _2,\dots,  \xi  _n) \in    D(0,4)^n$ let $F_\xi$ denote the partial function $ D(0,4)\ni z \mapsto F(z, \xi  _2,\dots,  \xi  _n)$, which is superharmonic (possibly $\equiv + \infty$). Denoting ${ \mathcal N}_1:= \{ \xi \in   D(0,4)^n\,;\, F_\xi \equiv + \infty  \}$, the partial gradient $ \nabla _1F( ., \xi  _2,\dots,  \xi  _n) $ is well-defined and its absolute value belongs to $L^1_{loc}(D(0,4))$ for $\xi \notin \mathcal n_1$. For $\xi \in \mathcal N_1$ we set the convention that $ | \nabla _1F( ., \xi  _2,\dots,  \xi  _n) |\equiv + \infty $.

We may define then the partial (local) maximal function ${ \mathcal M}^{2}_{ | \nabla _1 F | }$ over $D(0,1)^n$. It is the positive {\sl everywhere defined} Borel function given by

$${ \mathcal M}^2_{ | \nabla _1F | } (z):=\sup_{0<r \leq 2}  \oint_{ D(z_1,r)}    | \nabla _1F( \xi  _1,   z_2,\dots, z_n) | \, d \xi  _1.$$

We can define  analogously the exceptional sets  ${ \mathcal N}_j$, the partial gradients $ \nabla_ j F $ and the respective maximal functions ${ \mathcal M}^2_{ | \nabla _j F | }$ for $j = 2,\ldots,n$. 

\vspace{5pt} 
{\bf B.} Let us denote $D=D(0,1)$. Fix $k>0$ and define, for $j=1,\dots,\, n$, the sets
$$A^{(j)}= \{ z \in    D ^n\,;\, { \mathcal M}^{2}_{  | \nabla _jF | }(z) \leq c_0k\,  \},$$ where $c_0>0$ is a small constant, chosen independently of $F$ and $n$ in such a way that  $F$ is $k$-Lipschitz over every $A^{(j)}  \bigcap \{z : z_k=z_k^0, \text{ for every } k \neq j\, \}$ (see Theorem \ref{thm:bojarski}). By Theorem \ref{thm:lipschitz-dim1} we have that
$$ \lambda _2((D^n \setminus A^{(j)}) \cap  \{ z : z_\ell=  \xi    _\ell {\rm \  pour \ }\ell\ne j\,  \}) \leq \frac {  c_1} {k^2 } F(  \xi  )^2$$
for every $ \xi   \in   D^n$, where $c_1$ is a constant (notice that the inequality remains true if $ \xi   \in   { \mathcal N}_j$) . Integrating in $ \xi   _\ell $, $\ell \ne j$, using Fubini's Theorem and Lemma \ref{lemma:integral-function-comparison} we get $$ \lambda _{2n}((D^n \setminus A^{(j)})  \leq { \frac {  c_2} {k^2 }}F(  \zeta    )^2$$
for every $ \zeta   \in D^n$. Setting $B_1:=  \bigcap_{j = 1,\cdots,n} A^{(j)}$ we get a Borel subset of $D^n$ such that (i) $ \lambda _{2n}(D^n \setminus B_1)  \leq { \frac {c_2n} {k^2 }} F(  \zeta    )^2$ for every $ \zeta   \in D^n$ and (ii) $F|_{B_1}$ is $k$-Lipschitz in each variable. Notice that the constant $c_2$ depends only on $n$.

Let $\alpha \in (0,1)$ a constant depending only on $n$ which will fix later. Throwing away the points of $B_1$ whose density relative to $D$ with respect to the second variable is $\leq \alpha$ we get by Lemma \ref{lemma:density-measure-slice} a new Borel set $B_2 \subset   B_1$ such that $ \lambda _{2n}(D^n\setminus B_2) \leq \frac {  c_3} {k^2 } F( \xi  )^2$ with a new constant $c_3=c_{3 }(n)$. Repeating this procedure with respect to the other variables we get $B_1\supset B_2 \supset \cdots \supset B_n = B$ such that $ \lambda _{2n}(D^n\setminus B_n) \leq { \frac {  c_4} {k^2 }}F( \xi  )^2$ with the property that all points of $B_p$ have density $\geq \alpha$ relatively to $D$ with respect to the first $p$ variables.

{\bf C.} Let us show now that with this construction, for every pair of points $u,\, v \in   B$ such that $ | u_1-v_1 |  \geq  | u_2-v_2 |  \geq \dots  \geq  | u_n-v_n | $, we have $ | F(u)-F(v) |  \leq nk\, | u-v | $.

To this end we show by  induction on $p$ that if  $u,\, v \in   B_{p}$, and $u_j=v_j$ for $j>p$ we have $ | F(u)-F(v) |  \leq 2pk\, | u-v | $.
For the sake of simplicity we treat the step from $p=n-1$ to  $p=n$, the proof for the general step being similar.
  
Denote $u=(u',u'') \in    \C^{n-1}\times  \C $, $v=(v',v'') \in    \C^{n-1}\times  \C$ and $ \rho = | v''-u'' |  \leq  | u'-v' | $. The sections $T_{u'}$ et $T_{v'}$ of $B_{n-1}$ in the fibers $\{u'\}\times D$ and $\{v'\}\times D$ have  $u''$ and $v''$ as points with density $ \geq  \alpha $. By Lemma \ref{lemma:density-intersection}, if $ \alpha $ is chosen to be close enough to $1$, there is a $w \in   D$ such that $(u',w) \in T_{u'}$, $(v',w) \in T_{v'}$ , $ | w-u'' |  \leq  \rho $ and $ | w-v'' |  \leq  \rho $. Therefore
$$ | F(u)-F(v) |  \leq  | F(u',u'')-F(u',w) | + | F(u',w)-F(v',w) | + | F(v',w)-F(v',v'') | .$$
We know that $F_{ \vert  B_1}$ is $k$-Lipschitz in the last variable and by the induction hypothesis, $F(.,w)$ is $2(n-1)k$-Lipschitz on $B_{n-1} \cap  \{ (z',z'')\,:\, z''=w \}$. Therefore
$$ | F(u)-F(v) |  \leq 2k \rho +2(n-1)k | u'-v' |  \leq 2nk\,  | u'-v' |\leq 2nk\,  | u-v | .$$ 

{\bf D.} We have thus shown that given $F$ separately superharmonic on $D(0,4) ^n$ and $k \geq 1$, there is a Borel subset $B \subset   D^n$ such that (i)  $ \lambda _{2n}(D^n\setminus B) \leq { \frac {  c} {k^2 }} F( \xi  )^2$, for every $ \xi   \in   D^n$,  $c=c(n)$  and (ii) $ | F(u)-F(v) |  \leq k\,  |u-v | $ for every $u$, $v \in   B$  satisfying $ | u_1-v_1 |  \geq \dots  \geq  | u_n-v_n | $. By choosing a permutation $ \sigma  $ of  $\{ 1,\dots ,n \}$   the Lipschitz condition still holds if $u,v$ in (ii) satisfy  $ | u_{ \sigma  (1)}-v_{ \sigma  (1)} |  \geq \dots  \geq  | u_{ \sigma  (n)}-v_{ \sigma  (n)} | $. Replacing $B$ by the intersection of the $n!$ sets obtained in this way we get a new  set $A \subset   D^n$ such that (i)  $ \lambda _{2n}(D^n\setminus A) \leq { \frac {  c'} {k^2 }} F( \xi  )^2$, for  $ \xi   \in   D^n$, $c'=c'(n)$ and (ii) $F_{  | B}$ is $k$-Lipschitz in $B$. Finally, the existence of the compact $L$ and the constant $C$ follows from the inner regularity of $ \lambda _{2n}$.

\end{proof}

\subsection{Auxiliary lemmas}

This first lemma tells us that for a positive separately superharmonic function $f$ on a domain $\Omega  $ of $   \C^n$, the quantities $\int_\omega   | f(z) | ^2\, d \lambda _{2n}(z)$ and $\inf  \{ f(z)^2 : z \in   \omega  \}$ are  in some sense equivalent and independent of the chosen  open subset $\omega  \subset    \subset   \Omega $. 

\begin{lemma} \label{lemma:integral-function-comparison} Let $\Omega $ be an open connected subset of $\C ^n$ and let $\omega  $, $\omega  '$ two non-empty open sets, relatively compact in $\Omega $. Then for every positive separately superharmonic function $f:\Omega  \to  \overline  \R _+$   and every $z \in   \omega '$ we have $$C^{-1}\,\inf  \{ f( \xi  )^2\,;\,   \xi   \in   \omega '\, \}  \leq \int _\omega   | f(x) | ^2\, d \lambda _{2n}(x)  \leq C\, f(z) ^2$$
where $C=C(\Omega ;\omega  ,\omega ') $ is a finite positive constant depending only on $\Omega $, $\omega $ and $\omega ' $. 
\end{lemma}

\begin{proof}

{\bf A. Case $\mathbf {n=1}$.} We may assume that $\Omega $ is bounded. For the right side inequality we can restrict ourselves, by replacing $f$ by its reduction -or r\'{e}duite- (with respect to $\Omega $)  over  $\omega$ $$R_F^\omega = \inf\{u : u \text{ is positive superharmonic}  \text{ in } \Omega \text{ and } u \geq F \text{ in } \omega \}$$ to the case where $f=G\mu  $ is the Green potential, in $\Omega $, of a finite measure supported in $\overline \omega $, that is $f(z) = \int G(z,z') d\mu(z')$. 

Since  $G(z,z') \leq C+{ \frac {  1} {2 \pi   }}\, \log ( { \frac {  1} {| z-z' |}}  )$ on $\omega \times \omega$ for a constant $C=C(\Omega ,\omega )$ we get $\int _\omega   | f( \xi  ) | ^2\, d \lambda _2( \xi  )\,  \leq  \, C  \|\mu   \|_1^2$ and since $G(z, \xi  ) \geq c=c(\Omega ,\omega, \omega ' )$ for $(z, \xi  ) \in   \omega\times \omega '$, the right side inequality follows. 

For the other inequality, we may replace $f$ by its reduction over $\omega ' $ and suppose $f=G\mu $ with $\mu  $ supported in $\overline \omega '$.  We can even assume, by approximating this reduced function, that $\mu $ is supported on a compact set  $L \subset    \subset   \omega '$ with non-empty interior. By  H\"older inequality and  balayage definition  we have 
$$ \lambda _2(\omega )\,\int _\omega   | f( \xi  ) | ^2\, d \lambda _2( \xi  )\,   \geq  \, [\int_\omega    \,  f( \xi  ) \, \, d \lambda _2( \xi  )\,]^2  = [\int f( \xi  )\, d\nu_{L}  ( \xi  )\,]^2,$$
  where $\nu _{L}$ is the balayage (or swept out measure) of  $1_{\omega}\,  \lambda _2$ over $L$, relatively to $\Omega $ (by definition $R_{G(1_\omega  \lambda _2)}^{\,L}=G(\nu _L)$). As $ |\nu _{L} |\ne 0$ we get the desired inequality: $\int _\omega   | f( \xi  ) | ^2\, d \lambda _2( \xi  )\,  \geq c(\omega ',\omega ,\Omega )\, \inf_{ \xi    \in    \omega' } f( \xi  )^2$.

{\bf B. Case $\mathbf{ n \geq 2}$.} We can easily reduce the problem to the case where $\omega$ and $\omega'$ are both open polydiscs $  \prod _{j=1}^nD_j$ with compact closure in $\Omega$.

Notice that when $\omega =\omega '$, the result follows with no greater difficulty: the first  inequality  is trivial and the  second one follows from  part A and Fubini's Theorem. In the case $\omega$ and $\omega'$ are distinct polydiscs it suffices to show that $\inf _\omega f \geq c(\Omega ,\omega ,\omega ')\, \inf _{\omega '} f$ which is the content of Lemma \ref{lemma:harnack} below.
\end{proof}

The following elementary lemma can be seen as an extension of the Harnack inequalities to superharmonic functions.

\begin{lemma} \label{lemma:harnack} Let  $\Omega $ be a connected open subset of $ \R^N$. Then for every $ \delta >0$, $   \eta >0$,   there is a  $c=c(\Omega , \delta,   \eta  )>0$ such that for every positive superharmonic function $f$ in $\Omega $ and every $z,\, z' \in   \Omega ( \delta ):= \{ m \in   \Omega \,;\, d(m;\partial \Omega ) \geq  \delta ,\;  | m |  \leq  \delta ^{-1}\,  \}$ we have 
$$f(z) \geq c\, \inf  \{ f( \xi  )\,;\,  |  \xi  -z' |  \leq    \eta  \}$$
\end{lemma}

{\sl Remark.} After perhaps diminishing $  \eta $, it suffices to treat the case where $  \eta  \leq  \delta /2$. On replacing  then $ \delta $ by  $  2\eta$, we may assume that $  \eta = \delta /2$.

\begin{proof} If $a,\, b \in   \Omega (\delta )$, $b \in   \Omega $ are such that $d(a,b) \leq  \delta /4$ we have, for $a' \in   B(a, \delta/4 )$,
$$f(a') \geq { \frac {  4^N} { 3^N\,v_N\, \delta ^N   }}\,\int _{B(a',3 \delta /4)} f(z)\, d \lambda _N(z)  \geq\, { \frac {  1} {3^N   }} \inf  \{ f(b')\,;\,  | b-b' |  \leq  \delta /4\, \}.$$
So if we denote  $m_{f, \delta }(z):=\inf  \{ f( \xi  )\,;\,  \xi   \in   \Omega ,\, |  \xi  -z |  \leq  \delta/4 \,  \}$, have 
$$m_{f, \delta } (a)  \geq  \, { \frac {  1} {3^N  }} \;m_ {f, \delta }(b).$$ 

 {\bf b)} Let $ \delta >0$. Fix a connected compact set  ${\mathbb K}  \subset   \Omega $ containing $\Omega ( \delta )$. We can cover ${\mathbb K} $ by a finite number of balls $B(a_j;  \delta' / 8)$, $1 \leq j \leq \ell$ where $ \delta '>0$ is chosen in such a way that ${\mathbb K}  \subset   \Omega ( \delta ')$. Since ${\mathbb K} $ is connected, any two points $m$, $m' \in   {\mathbb K} $ can be joined by a ${ \frac {   \delta '} {4 }}$-chain $ \{ m;a_{i_1};\dots;a_{i_k};\dots;a_{i_\nu };m' \}$ of points of ${\mathbb K} $ (with $\nu  \leq \ell $). From the part  a) we get $f(m) \geq  c\, \inf  \{ f(z)\,;\,  | z-m' |  \leq  \delta '/4\, \}$.
 
 \end{proof}

\begin{lemma}  \label{lemma:lipschitz-gluing} Let $\overline  D_n(a_j,r):=\prod_{p=1}^n \overline  D(a_j^p,r)$, $1 \leq j \leq \ell$, a sequence of closed polydiscs of same radius in $\C ^n$ whose union ${ \mathbb L} _0$ is connected and  let $F:  { \mathbb L}_1:=\bigcup_{1 \leq j \leq \ell}    \overline  D_n(a_j,2r) \to  \R $ be a real function. Suppose that there exists a constant $c>0$  such that for every  $k>0$ and every $j=1,\dots, \ell$ there is a compact subset ${ \mathbb K} _{k,j} \subset   \overline  D_n(a_j,2r)$ such that (i) $F_{ | { \mathbb K} _{k,j}}$ is $k$-Lipschitz and (ii)  $\lambda _{2n}(\overline  D_n(a_j,2r)\setminus \mathbb K _{k,j}) \leq {\frac  { c} {k^2}} $.

Then there exists  $c'>0$ and for every $k>0$ a compact subset ${ \mathbb K} _k \subset   { \mathbb K} $ such that:

(a) $F_{ | { \mathbb K}_k }$ is $k$-lipschitz

(b)  $\lambda _{2n}({ \mathbb L} _0\setminus { \mathbb K} _{k}) \leq {\frac  { c'} {k^2}} $.

Moreover, we can chose $c'$ as depending only on $c$ and on the sequence  $ \{ D_n(a_j,r) \}_{1 \leq j \leq \ell}$.

\end{lemma}

\begin{proof} Notice first that it suffices to show that the properties (a) and (b) hold for $k$ bigger than some $k_0 = k_0(n,\ell,c,r )$ since the case of arbitrary $k$ will follow by replacing $c'$ by a bigger constant.

Set $\omega _{k,j}=\overline  D_n(a_j,2r)\setminus { \mathbb K} _{k,j}$, ${ \mathbb K} ^{(1)}_k= { \mathbb L} _1\setminus  \cup_j \omega _{k,j}$, ${ \mathbb K} ^{(0)}_k= { \mathbb L} _0\setminus  \cup_j \omega _{k,j}={ \mathbb K}^{(1)}_k \cap   { \mathbb L} _0   $. By assumption, $ \lambda _{2n}(\omega _{k,j}) \leq {\frac  { c} {k^2}}$, $\omega _{k,j} \subset   \overline  D_n(a_j, 2r)$,  and $ | f(m)-f(m') |  \leq  k | m-m' | $  for $m,\,m' \in   D_n(a_j,2r) \cap    { \mathbb K}^{(1)}_k$, $1 \leq j \leq \ell$.
 
In particular, if  $\overline  D_n(a_{j_1},r) \cap    \overline  D_n(a_{j_2},r)\ne  \emptyset  $, and $k$ is big enough ($k \geq k_0(n;r,c)$), then ${ \mathbb K}^{(0)}_k \cap    \overline  D_n(a_{j_1},2r) \cap     \overline  D_n(a_{j_2},2r)\ne  \emptyset  $ and hence $ | f(m_1)-  f(m_2) | \leq 2 n k r$ for $m_1 \in   { \mathbb L} _0 \cap    \overline  D_n(a_{j_1},r)$,   $m_2 \in   { \mathbb L} _0 \cap    \overline  D_n(a_{j_2}, r)$.  From the connectedness of  ${ \mathbb L} _0$ we get that 
$  | f(m)-f(m') |  \leq 2n\ell kr$ for $m$, $m' \in    { \mathbb K}^{(0)}_k$.
 
 Therefore, if $m$, $m'$ are points in ${ \mathbb K}^{(0)} _k$ that do not belong to the same polydisc $\overline  D_n(a_j,2r)$, we have$ | m-m' |  \geq 2r$ and (for $k \geq k_0$)
 $$ | f(m)-f(m') |  \leq 2n\ell k r = 2n\ell\, k{\frac  { r} { | m-m' | }} \times  | m-m' |  \leq c_2\,k\, | m-m' | ,$$
where $c_2=n\ell $.

We see then that (a) and (b) hold for $k \geq c_2\,k_0$ if we set  ${\mathbb K} _k={\mathbb K} _{c_2^{-1}k}^{(0)}$ and $c'=(c_2)^2\,c\ell$.

\end{proof}

In the following we say that a Borel subset $A$ of the disk $D(0,1)$ is of density $\geq  \alpha $ at $z$ relatively to  $D(0,1)$ if $ \lambda _2(A \cap D(z, \rho ) \cap D(0,1)) \geq  \alpha \,  \lambda _2( D(z, \rho ) \cap D(0,1))$, for all $ \rho >0$. Notice that this is much stronger than the usual notion of a density at $a$ larger than $\alpha $.
 
\begin{lemma} \label{lemma:density-intersection} Let  $z$, $z' \in   D(0,1)$, $r= | z-z' | $, and $A,\, A' \subset   D(0,1)$. Then there is a constant $\alpha_0$ with the following property: if relatively to $D(0,1)$, $A$ and $A'$ are  of density $\geq \alpha_0$ at $z$, $z'$ respectively,  then $A \cap A' \cap D(z,r) \cap    D(z',r)$ is non-empty.
\end{lemma}

\begin{proof} If $A \cap A' \cap D(z,r) \cap    D(z',r)= \emptyset  $, then either $ \lambda _2(A \cap D(z,r) \cap D(z',r)) \leq { \frac {  1} {2 }} \lambda _2(D(z,r) \cap D(z',r) \cap D(0,1))$ or  $ \lambda _2(A' \cap D(z,r) \cap D(z',r)) \leq { \frac {  1} {2 }} \lambda _2(D(z,r) \cap D(z',r) \cap D(0,1))$. In the first case  $ \lambda _2(D(z,r) \cap D(0,1)\setminus A) \geq { \frac {  1} {2 }} \lambda _2(D(z,r) \cap D(z',r) \cap D(0,1)) \geq  c_0 \lambda _2(D(z,r) \cap D(0,1))$ for some constant $c_0$. This means that $A$ is of density $ \leq 1-c_0$ in $z$ relatively to $D(0,1)$. A similar argument applies in  the second case using $A'$, $z'$ instead of $ A$, $z$. Hence, it suffices to take $ \alpha _0=1-c_0$.

\end{proof}

\begin{lemma} \label{lemma:density-measure} Let $A \subset   \R ^N$ be a Borel set contained in a cube (or an open ball) $C_0$ and let  $ \alpha  \in   (0,1)$. Denote  $A_N( \alpha )$ the set of points of $A$ where $A$ is of density $\geq \alpha$ relatively to $C_0$. Then $A_N( \alpha )$ is Borel-measurable and
 \begin{align}    \lambda _N(C_0\setminus   A_N( \alpha ))
 & \leq  c(N, \alpha ) \lambda _N(C_0\setminus   A)\, \nonumber
\end{align}
for a finite constant $c(N, \alpha )>0$.
 
\end{lemma}

\begin{proof} We can cover $A\setminus  A_N( \alpha )$ by balls  $B_x=B(x,r_x)$, $x \in   A\setminus  A _N( \alpha )$ satisfying 
$$ \lambda _N (C_0 \cap    B_x\setminus   A)>(1- \alpha ) \, \lambda _N(C_0 \cap B_x).$$
By the Besicovich's covering theorem (cf.\ \cite{book-mattila}, p.\ 30) there is an integer $\nu _N$ depending only on $N$ such that we can extract a countable sub-family $ \{ B_{x_j} \}_{j \geq 1}$ of balls that are $\nu _N$ to $\nu _N$ disjoints and still cover $A\setminus  A_N( \alpha )$. It follows that
 \begin{align}  \lambda _N(A\setminus A_N( \alpha ))  \leq  \sum_j\,    \lambda _N(C_0 \cap   B_{x_j}) 
 & \leq    \sum_j\,  (1- \alpha )^{-1}\lambda _N(C_0 \cap    B_{x_j} \setminus A)\nonumber 
\end{align}
and so $$\lambda _N(A\setminus A_N( \alpha ) )  \leq \,  \nu _N\, (1- \alpha )^{-1}\,   \lambda _N(C_0\setminus A).$$

For the  measurability of $A_N( \alpha )$ it is enough to observe that
$$A_N( \alpha )=  \bigcap _{r \in   { \mathbb Q} _+^*,\,  \beta  \in   { \mathbb Q} _+^*,\,  \beta < \alpha }  \{ x \in   A :  \lambda _N(A \cap B(x,r) \cap C_0) >  \beta \,  \lambda _N( B(x,r) \cap C_0)\, \}$$
and that every $A_N( \beta ,r):= \{ x \in   A : \lambda _N(A \cap B(x,r) \cap C_0) >  \beta \,  \lambda _N( B(x,r) \cap C_0)\, \}$ is relatively open in $A$.

\end{proof}

\begin{lemma} \label{lemma:density-measure-slice} We keep the notations and assumptions of Lemma \ref{lemma:density-measure} and fix a decomposition $\R ^N=\R ^m\times \R ^p$. Let $ A_N(m, \alpha )$ be the set of points $x=(x',x'')$ in $   A\,\,$ such that with respect to the slice $C_0 \cap     \{ x' \}\times \R ^p$,  $A_{x'}= \{ y ''\,;\, (x',y'') \in   A   \}$ is of density $\geq  \alpha $. Then $A_d(m, \alpha )$  is Borel-measurable  and there is a constant $C(m,\alpha)$ such that
$$ \lambda _N(A\setminus  A_d(m, \alpha )) \leq \, C(m, \alpha ) \, \lambda _N(C_0\setminus A) $$
\end{lemma}

\begin{proof}
The proof of the measurability follows the same lines as above in the proof of Lemma  \ref{lemma:density-measure} and the inequality follows from Lemma \ref{lemma:density-measure} and Fubini's Theorem.
\end{proof}

\bibliography{refs-ancona-kaufmann}
\bibliographystyle{alpha}
\end{document}